\documentclass[11pt]{article}
\usepackage{graphicx}
\usepackage{amssymb}
\usepackage{epstopdf}
\usepackage{amsthm}   
\usepackage[T1]{fontenc}

\newcommand{\bC}{{\Bbb C}}
\newcommand{\bP}{{\Bbb P}}
\newcommand{\bZ}{{\Bbb Z}}

\newcommand{\bR}{{\Bbb R}}
\newcommand{\bN}{{\Bbb N}}
\newcommand{\bA}{{\Bbb A}}

\newcommand{\spec}{{\rm{Spec\ }}}

\newcommand{\ord}{{\rm{ord}}}

\newcommand{\ok}{{\overline{K}}}

\renewcommand{\oe}{{\overline{E}}}
\newcommand{\oa}{{\overline{a}}}
\newcommand{\ini}{{\rm in}}

\renewcommand{\o}[0]{{\mathcal O}}

\renewcommand{\a}{{\frak{a}}}
\renewcommand{\b}{{\frak{b}}}

\newcommand{\mld}{{\rm{mld}}}

\newcommand{\we}{{\widetilde{E}}}

\newcommand{\m}{{\frak{m}}}
\renewcommand{\char}{{\rm{char\ }}}
\newcommand{\mod}{{\rm{mod}\ }}

\def\to {\longrightarrow}

\newtheorem{thm}{Theorem}[section]

\newtheorem{lem}[thm]{Lemma}
\newtheorem{cor}[thm]{Corollary}

\newtheorem{prop-def}[thm]{Proposition-Defintion}

\theoremstyle{definition}

\newtheorem{defn}[thm]{Definition}

\newtheorem{exmp}[thm]{Example}
\newtheorem{conj}[thm]{Conjecture}
 
\newtheorem{rem}[thm]{Remark}

\title{A bound of the number of weighted blow-ups to compute \\
the minimal log discrepancy
for smooth 3-folds}
\author{Shihoko Ishii}
\begin{document}
\date{}

\maketitle
\footnote{The author is partially supported by JSPS 19K03428}

\begin{abstract}
\noindent     
We study a pair consisting of a smooth 3-fold defined over an algebraically closed field 
    and a ``general'' $\bR$-ideal.
    We show that the minimal log discrepancy (``mld" for short) of
     every such a pair 
    is computed by a prime divisor obtained by at most two weighted blow-ups.
    This  bound is regarded as a weighted blow-up version of Musta\c{t}\v{a}-Nakamura's conjecture.
    We also show that if the mld of such a pair is not less than 1, then it is computed by at most one weighted blow-up.
    As a consequence, ACC of mld holds for such pairs.
\end{abstract}

\section{Introduction}

\noindent
Throughout this paper, the base field $k$ of varieties is an algebraically closed field of
arbitrary characteristic.
We study  pairs $(A,\a)$ consisting of a smooth variety $A$ of dimension $N>1$ and 
a ``$\bR$-ideal" $\a$ which means $\a=\a_1^{e_1}\cdots\a_r^{e_r}$,
where $\a_i$'s are non-zero coherent ideal sheaves on $A$ and $e=(e_1,\ldots,e_r)\in {\bR^r_{>0}}$.
We fix a closed point $0\in A$.

The minimal log discrepancy (``mld" for short) $\mld(0; A,\a)$ is an important invariant to measure 
the singularity of the pair $(A,\a)$ at 0 and plays important roles in birational geometry.
We consider every prime divisor over $A$ with the center at 0 and construct a ``good model"
of the divisor to approximate the mld.
The prototype is as follows:

\begin{thm}[\cite{kawk1},\cite{ip}]\label{prototype} Assume $N=2$.
      For every prime divisor $E$ over $A$ with the center at $0$, there exists a prime divisor
      $F$  obtained by one weighted blow-up with the center at $0$ satisfying
      $$a(E; A,\a)\geq a(F;A,\a),$$
      for every $\bR$-ideal $\a$ such that $a(E; A,\a)\geq0$.
\end{thm}      
          The inequality in the theorem implies that $F$ is a better divisor  to approximate the $\mld$.
       Therefore the theorem states that every prime divisor over $A$ with the center at 0 has a better divisor which  is
     obtained in a simple procedure.
     Here, we note that $F$ is constructed from $E$ and does not depend 
     on the choice of an $\bR$-ideal $\a$.

     Actually, in the paper \cite{kawk1}and \cite{ip}, the main theorem is not stated
in this form, 
but its proof shows  Theorem \ref{prototype}.
     The paper \cite{kawk1} is for $\char k=0$, 
     and the paper \cite{ip} is for
     $\char k=p>0$ and
     the main statements of both papers are in the following form:

\begin{cor}[\cite{kawk1},\cite{ip}]\label{kawcor} Assume $N=2$.
    Then, for every pair $(A,\a)$, the minimal log discrepancy $\mld(0;A,\a)$ is
    computed by a prime divisor obtained by one weighted blow-up.

\end{cor}

   The corollary follows from the theorem immediately.
  See, for example, the proof of Corollary \ref{maincor} in the section 5.

   When we consider the case $N=3$, we can see that one weighted blow-up is not sufficient
    to obtain a prime divisor computing the mld (see Example \ref{not1}).
   On the other hand,  in the example we can also show that the mld is computed 
    by a prime divisor obtained by two weighted blow-ups.
   So it is natural to expect the following conjecture:

\begin{conj}\label{conj} Assume $N\geq3$.
      For every prime divisor $E$ over $A$ with the center at $0$, there exists a prime divisor
      $F$  centered at $0$ obtained by at most $N-1$ weighted blow-ups satisfying
      $$a(E; A,\a)\geq a(F;A,\a),$$
      for every $\bR$-ideal $\a$ such that $a(E; A,\a)\geq0$.
\end{conj}
As an immediate consequence of the conjecture, we obtain the following:

\begin{conj}[Corollary of Conjecture \ref{conj}]\label{corconj} Assume $N\geq3$.
    Then, for every pair $(A,\a)$, the minimal log discrepancy $\mld(0;A,\a)$ is
    computed by a prime divisor obtained by at most $N-1$ weighted blow-ups.
\end{conj}

One of the motivations of the conjectures is that it is considered as a ``weighted blow-up version"
of Musta\c{t}\v{a}-Nakamura Conjecture (MN-Conjecture for short):

\begin{conj}[MN-Conjecture \cite{mn}]\label{mn} Fix $N$ and the exponent $e$ of $\bR$-ideals.
   Then, there exists a number $\ell_{N,e}\in \bN$ depending only on $N$ and $e$
   such that for any $\bR$-ideal $\a$ with the exponent $e$ 
  the minimal log discrepancy $\mld(0;A,\a)$ is
    computed by a prime divisor obtained by at most $\ell_{N,e}$ times blow-ups.
    Here, the blow-up means the ``usual blow-up", i.e., blow-up with the center at an irreducible reduced
    closed subset.
\end{conj}   
   
If this conjecture holds, then ACC Conjecture for these pairs holds (\cite{mn}), so it seems
to be a significant conjecture. 
   On the other hand, MN-Conjecture is equivalent to a reasonable conjecture on arc spaces (\cite{ii}),
   so it makes sense to study it.\\
    Note that MN-Conjecture requires to fix an exponent $e$, while the weighted blow-up versions
(Conjecure \ref{conj}, \ref{corconj})
do not require it.
Assume Conjecture \ref{conj} holds, 
it is also an interesting question whether the weights of the blow-ups can be bound uniformly in terms of exponents.
This will strengthen the MN-Conjecture.

    Another motivation of Conjecture \ref{conj} is for the project to bridge between positive 
    characteristic and characteristic 0 (\cite{ii}).
    In \cite{ii}, we have:
 \begin{lem} Let $\a$ be an $\bR$-ideal on a smooth variety $A_k$ over $k$ ($\char k=p>0)$
         and $E$ a prime divisor over $(A_k,0_k)$  computing $\mld(0_k;A_k,\a)$.
      
      If there exists an $\bR$-ideal $\widetilde\a$ on a smooth variety $A_\bC$ over $\bC$
     and a prime divisor $\widetilde{E} $ over $(A_\bC, 0_\bC)$, where $0_\bC\in A_\bC$
     such that 
     \begin{enumerate}
         \item $\widetilde\a (\mod p)=\a$  (see \cite{ii} for the definition of (\mod p))
         \item $a(\widetilde{E}; A_\bC, \widetilde\a)\leq a(E; A_k, \a)$,
     
     \end{enumerate}
     then, $\mld(0_\bC; A_\bC, \widetilde\a)=\mld(0_k; A_k, \a)$.
 \end{lem}   
 \begin{rem}
     In particular, if such  $\widetilde\a$ and $\we$ exist  for every $\a$  and  $E$
     and assume that  $\mld(0_k;A_k,\a)$
     is computed by a divisor, then the set of $\mld(0_k; A_k, \a)$'s is contained in 
     the set of $\mld(0_\bC; A_\bC, \b)$'s.
     Therefore, if we fix the exponent $e$ and the dimension $N$ of $A_k$, then the number of the values $\Lambda_e:=\{\mld(0_k, A_k, \a) \mid \a \ \mbox{is\ a \ $\bR$-\ ideal
     \ with\ the exponent \ } e\}$
     is finite for $\char k>0$, because it is proved to be finite in characteristic $0$ by \cite{kawk}.
     Similarly,
     if ACC holds in characteristic 0, then it also holds in positive characteristic. 
\end{rem}

    Now, the problem is to construct appropriate $\widetilde{E}$ and $\widetilde\a$
    for  given $E$ and $\a$.
    If Conjecture \ref{conj} holds, we can reduce this problem to a divisor $F$ of special type
    (i.e., obtained by at most $N-1$ weighted blow-ups), which seems easier to handle.   
    
    The main results of this paper are the following:
 \begin{thm}\label{main} Assume $N=3$.
      For every prime divisor $E$ over $A$ with the center at $0$, there exists a prime divisor
      $F$  centered at $0$ obtained by at most two weighted blow-ups satisfying
      $$a(E; A,\a)\geq a(F;A,\a),$$
      for every   ``general''  $\bR$-ideal $\a$ for $E$ such that $a(E; A,\a)\geq0$. 
 \end{thm}   
 
   The terminology ``general'' will be defined in Definition \ref{defofgeneral}.
   The weighted blow-ups will be constructed by ``squeezed'' blow-ups (see, Definition \ref{def.st.blowup}) depending only on $E$
   and it works for every general ideal.
   Here, ``general'' is necessary, because there exists an example of non-general ideal such that two  squeezed blow-ups do not give the 
   required divisor in the theorem (cf. Example \ref{3times}).
   But it does not give a counter example for Conjecture \ref{conj}, indeed for the example there exists another sequence of weighted blow-ups
   to obtain the required divisor (see, also Example\ref{3times}).

 As a corollary we obtain: 
 \begin{cor}\label{maincor} Assume $N=3$.
    Then, for every  pair $(A,\a)$ with a ``general'' $\bR$-ideal $\a$, the minimal log discrepancy $\mld(0;A,\a)$ is
    computed by a prime divisor obtained by at most two weighted blow-ups. 
 \end{cor}
 
  It is known as the Zariski's sequence 
  that every prime divisor $E$ over $A$ with the center at $0$ is obtained by successive usual blow-ups from $A$, 
 such that the centers of blow-ups are the center of $E$ on each step (\cite[VI, 1.3]{Ko} ).
 The following corollary shows that in some cases, we obtain the two weighted blow-ups to compute the mld
 by just looking at the center of the second blow-up in the Zariski's sequence.
 
 \begin{cor}[Corollary \ref{111}]\label{usual} Assume $N=3$.
     Let $E$ be a prime divisor over $A$  computing $\mld(0; A, \a)$ for a pair $(A,\a)$.
     Let  $A_1\to A$ be the first usual blow-up with the center at $0$ in the Zariski's sequence.
     Assume that the center $C\subset A_1$ of $E$  is a curve of degree $\geq 2$ in 
     the exceptional divisor $E_1\simeq \bP^2$.
     Then a weighted blow-up which is called `` squeezed blow-up'' at $C$ gives a divisor computing $\mld(0; A, \a)$.   
\end{cor}
     Note that in this case the first blow-up is also a  squeezed blow-up.
    Example \ref{not1} is just in this case.
    In Section 5, we show  a more general corollary.
    On the other hand, if we restrict to the case $\mld\geq 1$, then we have the following:
   
 \begin{thm}\label{mld1} Assume $N=3$.
    Then, for every  general pair $(A,\a)$ with  $\mld(0;A,\a)\geq 1$, 
    the minimal log discrepancy  is
    computed by a prime divisor obtained by one weighted blow-up. 
 \end{thm}
 
 \begin{cor}\label{mn3}Assume $N=3$. In
$$\Lambda=\{(A,\a) \mid  \ \mld(0; A,\a)\geq 1\ \mbox{with\ general }\ \a \}$$
Musta\c{t}\v{a}-Nakamura Conjecture holds 
and also ACC Conjecture holds for $\char k\geq 0$.
Here, ACC Conjecture means that the set of $\mld(0; A,\a)$ for the pairs in the subset 
 $\Lambda_J\subset \Lambda$ consisting of $\bR$-pairs with the exponents in $J\subset \bR_{>0}$
 satisfies Ascending Chain Condition.
 Here, $J$ is a DCC set.
\end{cor}

    The corollary follows from Theorem \ref{mld1} in the same way as in  the proof of \cite[Corollary 1.6]{ip},
    since the mld is computed by one weighted blow-up.

This paper is organized as follows: in Section 2 we prepare  basic terminologies 
which will be used in this paper.
In Section 3 we discuss about weighted blow-up at a (not necessarily closed)
smooth point and basic formula on weighted projective space, 
that is the exceptional divisor appearing in a weighted blow-up.
In Section 4 we construct an appropriate regular system of parameter (RSP for short)
with the weight, in order to make a weighted blow-up.
In Section 5 we give the proofs of the main results.

{\bf Acknowledgement} The author would like to thank Kohsuke Shibata, Lawrence Ein, Masayuki Kawakita and Yuri Prokhorov for
  useful discussions.
 The author expresses her hearty thanks to the referee for the numerous constructible comments to improve the readability greatly
 of the paper.

\section{Preliminaries}

Let $A$ be an $N$-dimensional smooth variety defined over an algebraically closed field $k$.
We fix a closed point $0\in A$.
\begin{defn} We call $E$ a {\sl prime divisor over } $A$,
  if there is a proper birational morphism $\varphi: A' \to A$ from a normal variety $A'$ on which 
  $E$ is an irreducible divisor.
  The generic point $P\in A$ of the image $\varphi(E)$ is called the {\sl center of } $E$ on $A$.
  In this case, we sometimes call $E$ a prime divisor over $(A,P)$.
  \end{defn}

\begin{defn} For a prime divisor $E$ over a non-singular variety $A$,
 let $\varphi: A'\to A$ be a proper birational morphism with normal $A'$
 such that  $E$ appears on $A'$.
  Let $k_E$ (or sometimes written as  $k_{E/A}$) be the coefficient of the relative canonical divisor $K_{A'/A}$
  at $E$  and $v_E$  the valuation defined by the prime divisor $E$.
  Here, note that $k_E$ ($k_{E/A}$) does not depend on the choice of $A'$.

   Let $\a$ be an $\bR$-ideal on $A$ as in the beginning of the first section and
   $e_i$'s are the exponents.
  The {\sl log discrepancy}
  of the pair $(A,\a)$ at $E$ is defined as 
  $$a(E; A, \a):=k_E-\sum_i e_iv_E(\a_i)+1$$
  and the {\sl minimal log discrepancy } of the pair at a closed point $0$ is defined as
  $$\mld(0; A, \a):=\inf \{a(E; A, \a)\mid E \ \mbox{prime\ divisor \ over\ }A \mbox{\ with\ the \ center\ at}\ 0\}$$
 
 \end{defn}
 
   It is known that for $N\geq 2$, either $\mld(0; A,\a)\geq 0$ or $\mld(0; A,\a)=-\infty$ holds.
   For $N=1$, we define $\mld(0;A, \a)=-\infty$ if the lefthand side is negative, by abuse of notation, 
   because it is convenient to describe the Inversion of Adjunction.
   
\begin{defn} We say that a prime divisor $E$ over $A$ with the center at $0$  {\sl computes} $\mld(0;A,\a)$
    
    if either  $a(E; A,\a)=\mld(0;A,\a)$ (when the right hand side is $\geq 0$)
     
     or $a(E; A,\a)<0$ (when the mld is $-\infty$).
    
\end{defn}   

\begin{rem}
     Assume  there exists a log resolution of the pair $(A, \a\m_0)$,
    where $\m_0$ is the maximal ideal defining $0\in A$.
     If $\mld(0;A,\a)\geq 0$,
   then, on every such resolution
    there is a prime divisor computing $\mld(0;A,\a)$.
   If $\mld(0;A,\a)=-\infty$ and $Z(\a)\subset A$ contains an irreducible component of codimension one,
   there may not exist a prime divisor computing the mld among the exceptional divisors appearing in a given log resolution
   (cf. \cite[Proposition 7.2]{EM}).
   But in this case, if we construct an appropriate log resolution of  $(A, \a\m_0)$ by taking more blowing-ups  from the given one, a prime divisor computing  $\mld(0;A,\a)$
   appears on that.
   Therefore, for $\char k=0$ or $N\leq 3$,  every pair $(A, \a)$ has a prime divisor computing  $\mld(0;A,\a)$,
   since there is a log resolution for every pair.
      \end{rem}

\section{Weighted blow-ups and weighted projective spaces}

    In this section $A$ is always a smooth variety of dimension $N\geq2$ defined over an algebraically closed field $k$
    and $P\in A$ is a (not necessarily closed) point.

\begin{defn}  Let $x_1,\ldots, x_c$ be an RSP of  a regular local ring $R$ with the algebraically closed residue field 
     and $w_1,\ldots,w_c$ be positive integers with $\gcd(w_1,\ldots,w_c)=1$.
     For $n\in \bN$, denote by $\mathcal I_n$ the ideal in $R$
     generated by all monomials $x_1^{s_1}\cdots x_c^{s_c}$
     such that $\sum_{i=1}^c s_iw_i\geq n$.
     The  {\sl weighted blow-up} of $\spec R$ with $wt_w(x_1,\ldots, x_c)=(w_1,\ldots, w_c)$ is 
     the canonical projection:
     $$\rm{Proj}_{A}(\oplus_{n\in \bN}{\mathcal I}_n)\to A:=\spec R.$$
     The exceptional divisor $E$ for the weighted blow-up is called {\sl a prime divisor
     obtained by a  weighted blow-up} of $A$ at $P$.  
     
     More generally, let $P\in A$ be a smooth point with the not-necessarily-algebraically closed residue field $K$.
     Let $\ok$ be the algebraic closure of the residue field of $\o_{A,P}$.
     A {\sl weighted blow-up of $A$ at the  point $P$ } is the canonical morphism induced from 
     a weighted blow-up $\overline{A}\to \spec \ok\widehat\o_{A,P}$ for some RSP $x_1,\ldots, x_c$  of $\ok\widehat\o_{A,P}$
     with $wt_w(x_1,\ldots, x_c)=(w_1,\ldots, w_c)$ for some $ (w_1,\ldots, w_c)\in \bZ^c_{>0}$,
     where $\ok\widehat\o_{A,P}$ is the extension of the formal power series ring $\widehat\o_{A,P}$ over $K$  to the one
     over $\ok$.
     Let $\oe$ be the prime divisor obtained by the weighted blow-up  $\overline{A}\to \spec \ok\widehat\o_{A,P}$.
     The prime divisor $E$ over $A$ with the center at $P$ corresponding to $\oe$   is called a {\sl prime divisor obtained 
     by a weighted blow-up}   of $A$ at $P$.
     Note that if $\oe $ gives a valuation $\overline v$ and the valuation ring $\o_{\overline v}$, the prime divisor $E$ corresponds to the valuation $v$ whose
     valuation ring is $K(A)\cap \o_{\overline v}$.
     
 \end{defn}

 Note that   weighted blow-ups are only  defined at  smooth   points.

Here, we show a 3-dimensional  example that the minimal log discrepancy is not computed  by a divisor obtained by
    only one weighted blow-up, but computed by a divisor obtained by two weighted blow-ups .

The following are well known, for example see \cite[Remark 2.6, Lemma 2.7]{kawk2}.   
 \begin{rem} 
   Let $P\in A$ be a point of a smooth variety with the residue field $K$.
   \begin{enumerate}
   \item The set of prime divisors over $A$ with the center at $P$  corresponds bijectively to
      the set of prime divisors over $\widehat A:=\spec \widehat\o_{A,P}$ with the center at the closed point.
      Moreover, if  prime divisors $E$ and $\widehat E$ correspond under the above bijection, then for every $\bR$-ideal $\a$ on $A$ we have $v_E(\a)=v_{\widehat E}(\a)$
      and also $a(E;A,\a)=a(\widehat E,\widehat A, \a\o_{\widehat A})$.
   
   \item
    Let  $K'\supset K$ be a field extension and $A':=\spec K'\widehat\o_{A,P}$.
    Then, there is a surjective map from the set of prime divisors over $A'$ with the center at the closed point 
    to the set of prime divisors over $A$ with the center at $P$.
    If  prime divisors $E'$ and $E$ correspond by the above surjective map, then
    it follows $a(E'; A', \a\o_{A'})=a(E; A, \a)$ for every $\bR$-ideal $\a$ on $A$.

   \end{enumerate}

 \end{rem}   
    
\begin{exmp}\label{not1}
  Assume $\char k\neq 2,5$.
Let $A:=\bA_k^3$ and $\a=(f)^{7/10}$,
where 
$$f=(x^2+y^2+z^2)^2+x^5+y^5+z^5.$$
Then, a divisor  computing $\mld(0;A,\a)=0$ is not obtained by
one weighted blow-up (\cite[Ex. 6.45]{ksc}).

On the other hand, there is a sequence of weighted blow-ups
$$A_2\stackrel{\varphi_2}\longrightarrow A_1\stackrel{\varphi_1}\longrightarrow A$$
where $\varphi_1$ is the usual blow-up at $0$ and $\varphi_2$ is a weighted blow-up
with weight $(1,2)$ at the generic point of the curve $x^2+y^2+z^2=0$ on $E_1=\bP_k^2$.
Here, $E_1$ is the exceptional divisor for $\varphi_1$.
The exceptional divisor $E_2$ for $\varphi_2$ computes $\mld(0;A,\a)=0$

\end{exmp}

   The following lemma for a weighted projective space with a special weight is used for our main results.
   The statement is easily generalized to higher dimensional case, 
   but for  simplicity of notation we state here only for 2-dimensional case.

\begin{lem}\label{bezout}
    Let $r\leq s$ be positive integers such that $\gcd(r,s)=1$.
    Let $g\in k[x_1,x_2,x_3]$ be a weighted homogeneous polynomial with respect to the
     weight $w=(w(x_1),w(x_2),w(x_3))=(r,r,s)$
    and $Q\in \bP_k(r,r,s)$ a closed point not contained in the coordinate planes,
    i.e., $Q\not\in (x_1\cdot x_2\cdot x_3=0)$.
    Let $\ell\in  k[x_1,x_2,x_3]$ be a weighted homogeneous polynomial of $\deg_w(\ell)=r$
    such that $\ell(Q)=0$.
    If $\ell \not| g$, then it follows
    $$r\cdot s\cdot\ord_Q(g)\leq 
    r\cdot s\cdot\ord_Q(g\mid _{L})\leq 
    \deg_w g,$$
     where  $L\subset \bP_k(r,r,s)$ is the divisor defined by  $\ell=0$ in  $\bP_k(r,r,s)$.

\end{lem}  

\begin{proof} 
      As $\ord_Qg\leq \ord_Q(g\mid _{\ell})$, the first inequality is trivial.
      We will show the second inequality.
      Let $G\subset \bP_k(r,r,s)$ be the subscheme
     defined by $g=0$ on $\bP_k(r,r,s)$.
     Let    $$\pi: \bP_k^2\twoheadrightarrow \bP(r,r,s),
     (X_1,X_2,X_3)\mapsto (X_1^r,X_2^r,X_3^s)=(x_1,x_2,x_3)$$ 
     be the canonical covering.
    Then, as $\pi^*L$ and $\pi^*G$ has no common irreducible components,  Bezout's theorem on $\bP^2$ implies
  \begin{equation}\label{bezout1}
  \pi^*L\cdot \pi^*G=\deg\pi^*\ell\cdot\deg\pi^*g=\deg_w\ell\cdot\deg_w g=r\cdot \deg_w g
    \end{equation}
   
 In case $\char k=0$ or $\char k=p>0$ and $p\not| r\cdot s$, the morphism $\pi$ is \'etale around
 $Q$. 
 Therefore, $\pi^{-1}(Q)$ consists of $r^2\cdot s$ closed points $\{Q_i\mid {i=1, \ldots, r^2\cdot s}\}$ whose analytic neighborhoods of
 $\pi^*G$ and $\pi^*L$ are isomorphic  to those of $G$ and $L$ at $Q$, respectively.  
 Then, by (\ref{bezout1}) we obtain
 $$r^2\cdot s\cdot\ord_Q( g\mid_{L})=\sum_{i=1}^{r^2s}\ord_{Q_i}(\pi^*g\mid _{\pi^*L})\leq  \pi^*L\cdot \pi^*G=r\cdot \deg_w g,$$
 which yields the required inequality.
   
   In case $p | r$, denote $r=p^e\cdot q$ ($\gcd(p,q)=1$).
   Then, the fiber $\pi^{-1}(Q)$ consists of $q^2\cdot s$ closed points, as a topological space.
   For a closed point $Q_i$ ($i=1,\ldots, q^2\cdot s$) in the fiber $\pi^{-1}(Q)$ we obtain
   $$\m_{Q}\o_{\bP^2}\subset \m_{Q_i}^{p^e},$$
   where $\m_{Q}$ and $\m_{Q_i}$ are the maximal ideals of $Q\in \bP(r,r,s)$ and of $Q_i\in \bP^2$, respectively.
   Let   $C\subset \bP^2$ be the subscheme with the reduced structure  of $\pi^*L$.
   Then, we have 
   $$\m_{L, Q}\o_C\subset \m_{C,Q_i}^{p^e},$$
   where $\m_{L,Q}$ and $\m_{C,Q_i}$ are the maximal ideals of $Q\in L$ and of $Q_i\in C$, respectively.
   Therefore,  for every $i=1,\ldots, q^2\cdot s$ it follows
   $$p^e\cdot\ord_Q (g\mid L)\leq \ord_{Q_i}(\pi^*g)\mid _C.$$
     Now, there are $q\cdot s$ points $Q_i$ 
   lying on $C$.
   Then, by Bezout's theorem on $\bP^2$ for $C$ and $\pi^*G$, we obtain
   $$q\cdot s\cdot p^e\ord_Q (g\mid _L)\leq q\cdot s\cdot\ord_{Q_i} (\pi^*g)\mid_ C\leq C\cdot\pi^*G=\deg_w g.$$
   Here noting that $q\cdot s\cdot p^e=r\cdot s$, this is the required inequality.
   
   In case $p | s$, the proof is similar.  
\end{proof}

\section{Squeezed Systems and Squeezed Blow-Ups}

Let $A$ be a variety of dimension  $N\geq2$ over an algebraically closed field $k$.
\begin{defn}\label{sque}
  Let $P\in A$ be a smooth point (not necessarily closed),
  $K$ the residue field, 
  and 
   $E$ a prime divisor over $A$ with the  center at $P$. 
   Denote the algebraic closure of $K$ by $\ok$.
   An RSP $\{x_1,\ldots, x_c\} $ of $\ok\widehat\o_{A,P}$ at  the closed point is called a 
   {\sl  squeezed system }  for $E$ at $P$,
   if  $v_i:=v_E(x_i)$ ($i=1,\ldots, c$) satisfy   
   \begin{enumerate}
   \item $v_1=\cdots=v_{c-1}\leq v_c$,
   \item $v_1:=\min\{v_E(x)\mid x\in \m\setminus \m^2\}$
   \item $v_c:=\max\{v_E(x)\mid x\in \m\setminus \m^2\},$
 
     \end{enumerate}
     where $\ok\widehat\o_{A,P}$ is the extension of the coefficient field $K$ of the formal power series ring $\o_{A,P}$
     to $\ok$, and 
      $\m \subset \ok\widehat\o_{A,P}$ is the maximal ideal.
     
  In this case,    $$v':=({v_1}',\ldots, {v_c}')=\frac{(v_1,\ldots,v_c)}{\gcd (v_1,\ldots,v_c)}$$
   is called a {\sl  squeezed weight}  for $E$ at $P$.
   
     Let $E$ and $v'=(v'_1,\ldots, v'_c) $ be as above.
   In this case, we call $E$ a prime divisor of squeezed type $v'$.
 \end{defn}
 
 Note that the  squeezed weight for $E$ is determined by a prime divisor but  squeezed system is not uniquely
 determined by the prime divisor $E$.

\begin{rem}
  For every $A, P$ and $E$  as in Definition \ref{sque}, there exists a  squeezed system of $\ok\widehat\o_{A,P}$.
  Indeed, it is obvious that there is $x_1\in\m\setminus \m^2$  
  such that $v(x_1)$ is the minimal value among 
  $\{v_E(x)\mid x\in \m\setminus \m^2\}$.
  Existence of the maximal  $v(x_c)$ among the set is proved by Zariski's subspace theorem
   (cf. \cite[(10.6)]{aby} ).
   Now, we extend $\{x_1,x_c\}$ to an  RSP $\{x_1,x_2,\ldots,x_c\}$ of $\o_{A,P}$.
   Here, if $v_E(x_i)> v_E(x_1)$ for $2\leq i \leq r-1$, replace $x_i$ by $x_1+x_i$.
   Then, we obtain a  squeezed system $\{x_1,x_2,\ldots,x_c\}$.
\end{rem}

Actually in \cite{kawk1} and \cite{ip}, the proofs of Theorem \ref{prototype}  show the following:
\begin{exmp}[Theorem \ref{prototype}]
      For every prime divisor $E$ over a smooth surface $A$ with the center at $0$
      such that $a(E; A, \a)\geq 0$ for an $\bR$-ideal $\a$ on $A$.
      Then, the exceptional divisor $E_1$ obtained by a squeezed blow-up for $E$ satisfies
         $$a(E; A,\a)\geq a(E_1;A,\a).$$
\end{exmp}

\begin{defn}\label{def.st.blowup}
  Let $A, P$ and $E$ as above and let $\{x_1,\ldots,x_c\}$ be a 
    squeezed system for $E$ and $v'=(v'_1,\ldots,v'_r)$ be the  squeezed weight.
   We call the weighted blow-up of weight $v'$ with respect to the coordinate system
   $\{x_1,\ldots,x_c\}$  a {\sl  squeezed blow-up} for $E$.
\end{defn}

\begin{rem}\label{not.st.rem}
   As in the definitions, a squeezed system is a RSP in the  local ring with extended coefficient field.
   A squeezed system is not in general a RSP of the original local ring $\o_{A,P}$.
\end{rem}
 
  \begin{exmp}\label{not.stan.exmp}
   Let $A_K:=\spec K[[y,z]]$ and  $A_\ok:=\spec \ok[[y,z]]$, where $\ok$ is the algebraic closure
  of $K$.
  Take an element $a\in \ok\setminus K$ and let $\phi\in K[T]$ be the minimal polynomial of $a$.
  Let $\varphi_1:A_1\to A_K$ be the usual blow-up at the closed point of $A_K$.
  Then the exceptional divisor $E_1$ is the projective line $\bP_K^1$ with the homogeneous coordinates $\{y,z\}$.
  Denote the homogenized polynomial of $\phi$ by $\Phi(y,z):=z^{\deg \phi}\phi(y/z)$.
  Take the blow-up $\varphi_2: A_2\to A_1$ with the center at the closed subscheme $C$ defined by the ideal $(\Phi(y,z))$
  on $E_1$.
    As the proper transforms of any curves defined by  linear forms $\ell=cy+dz=0$ ($c,d\in K$) on $A_1$ do not intersect to $C$,
  it follows $v_{E_2}(\ell)=1$.
  Therefore, every RSP $\{ f_1, f_2\}$ of $K[[y,z]]$  satisfies $v_E(f_1)=v_E(f_2)=1$.
   
  On the other hand,  take the base change $\psi: A_\ok\to A_K$ by the field extension $\ok\supset K$.
  Let $z':=y-az\in \ok[[y,z]]$. Then, the proper transform of the curve defined by $z'=0$ contains the point $(a:1)\in \bP_\ok^1=\oe_1$
  where $\oe_1$ is the exceptional divisor of the blow-up at the closed point of $A_\ok$.
  As $(a:1)\in \oe_1$ satisfies $\Phi(y,z)=0$, the proper transform of $z'=0$ intersects the center of the second blow-up
  induced from $\varphi_2$. 
  One can see that $v_E(z')>1$, and therefore a squeezed system cannot be taken from $K[[y,z]]$.

  \end{exmp}

        Now we are going to define ``general" ideal.  
\begin{defn}\label{defofbad}
     Let $E$ be a prime divisor over  $A$ of squeezed type $(v'_1,v'_2, v'_3)$ 
    (note that $v'_1=v'_2$) and let $E_1$ be  the exceptional divisor obtained by the  squeezed blow-up
    with respect to a  squeezed system
     $\{x_1,x_2,x_3\}$.

     An irreducible curve $B\subset E_1=\bP(v'_1,v'_2, v'_3)$ with the following properties
     is called  a {\sl bad curve} for $E$  on $E_1$.
\begin{enumerate}
    
   \item[(1)]
        $B$ is a curve of degree $v'_1$ with respect to $(v'_1,v'_2, v'_3)$.
        (In the discussions on a weighted projective space, ``degree'' always means degree with respect to 
        $(v'_1,v'_2, v'_3)$, and it is sometimes denoted by $\deg_{v'}$.)
       \item[(2)]
           $B$ contains the center of $E$.
 \end{enumerate}
 
 \end{defn}

\begin{lem}\label{non-exist-bad}
Under the setting of Definition \ref{defofbad}, the following hold:
\begin{enumerate}
\item[(i)]
      A bad curve does not always exist.
      More precisely a bad curve does not exist if and only if one of the following holds:
\begin{enumerate}
       \item The  squeezed weight is $(1,1,1)$, or
       \item  The  squeezed weight  $(v'_1,v'_2,v'_3)$  satisfies  $v'_1<v'_3$
  and the center of $E$ on $A_1$ is a curve of  $\deg_{v'}\ >v'_1$ on $E_1\simeq \bP(v'_1,v'_2,v'_3)$, or 
       \item  $E=E_1$.
\end{enumerate}
 \item[(ii)]
 If a bad curve exists, then it is unique in $E_1$.
 \end{enumerate}
 \end{lem}
 
 \begin{proof}
   It is clear that if $E=E_1$, then the center of $E$ on $E_1$ is the generic point, so there is no bad curve on $E_1$.
   We exclude this trivial  case in the following discussions.
  In case the  squeezed blow-up is the usual blow-up, then the exceptional divisor does not have
a bad curve.
    Because if $B$ is a bad curve, it is defined by linear form $\ell=\sum_i a_iX_i=0 $ with $a_3\neq 0$, where $\{X_1,X_2,X_3\}$
    is the projective coordinate system on $E_1=\bP^2$ corresponding to the  squeezed system $\{x_1,x_2,x_3\}$ on $\o_{A,0}$.
    This is a contradiction to the fact that $(1,1,1)$ is the  squeezed system , as we obtain another RSP $\{x_1,x_2, \ell(x_1)\}$ such that \begin{equation}\label{oftenuse}
    v_E(x_1)<v_E(\ell(x_i)). 
  \end{equation}
  Here, we give the proof of this inequality, as this kind of discussion is used frequently in this paper.
    
  Let $\varphi_1:A_1\to A$ be the  squeezed blow-up and $\psi :\widetilde A\to A_1$ a birational morphism on which 
  $E$ appears.
  Denote the composite $\varphi_1\circ \psi $ by $\varphi$.
  Let $D$ be the proper transform of $Z(\ell(x_i))\subset A$ in $A_1$, then  $D\cap E_1 $ contains the center of $E$
  on $A_1$ by the assumption.
 Note that we can express 
     $$
    (\varphi_1^*\ell(x_i))=r E_1+D,\ \ (r=v_{E_1}(\ell(x_i)))
    $$
   
   Here, we remind us that $v_E(\ell(x_i))$ is the coefficient of the divisor 
     $(\varphi^*\ell(x_i))=\psi^*(rE_1+D)$ at   the component $E$.
     The center of $E$ on $A_1$ is contained in $D$, therefore the contribution from $\psi^*(D)$ to
     $v_E(\ell(x_i))$ is positive. 
     Therefore, $v_E(\ell(x_i))> r v_E(E_1)=v_{E_1}(\ell(x_i))v_E(E_1)=v_E(x_1)$
     This shows the inequality (\ref{oftenuse}).

For the case where $E_1$ is an exceptional divisor of a  squeezed blow-up with respect to $(v'_1,v'_2, v'_3)$ with $v'_1<v'_3$,
if the center $C$ of $E$ on $E_1$ is a curve of degree $>v'_1$, then there is no bad curve.
Because, a curve of degree $v'_1$ cannot contain a curve of degree $>v'_1$.
This gives the proof of ``if" part of (i).

 Assume a bad curve exists on $E_1$.
When the center of $E$ on $E_1$  is a curve, then it should coincide with the bad curve by the definition, 
therefore the center should be of degree $v'_1$.
When the center of $E$ on $E_1$ is a closed point $P$,  then a bad curve should contain $P$.
Express the point $P$ by the homogeneous coordinates $(a,b,c)$ with $a,b,c\in k$.
Then a curve of degree $v'_1$ containing $P$ is defined by $bX_1-aX_2=0$.
Now we obtain  the uniqueness of the
bad curve on $E_1$.
This completes  the proof of ``only if'' part of (i) and the proof of (ii).
 \end{proof}


\begin{defn}\label{defofgeneral}
   Let $E$ be a prime divisor over a smooth variety $A$ with the center at a closed point $0$. 
   An $\bR$-ideal $\a$ is called {\sl general for} $E$ 
   if   there exists a    squeezed blow-up  $ A_1\to A$ for $E$ with the exceptional divisor $E_1$
   satisfying the following:
    \begin{enumerate}
 \item[(1)]  $\ord_B\a_{A_1}\o_{E_1}\leq 1$,  
  where $B$ is the bad curve on $E_1$ and $\a_{A_1}$ is the weak transform of $\a$ at $A_1$.
  If there is no bad curve on $E_1$, then we account it as the inequality  automatically holds.
  
    \item[(2)]  In addition, if $a(E;A,\a)<a(E_1;A,\a)$ and the center $P$ of $E$ on  $A_1$ is  a smooth
     closed point,
           then there exists a  squeezed blow-up $A_2\to A_1$ for $E$ at $P$.
            Let $E_2$ be the exceptional divisor.
            Then,  $\ord_{B'}I_L\a_{A_2}\o_{E_2}\leq 1$,   
             where $B'$ is the bad curve on $E_2$, $\a_{A_2}$ is the weak transform of $\a$ at $A_2$
             and $I_L$ is the defining ideal of the intersection $L:=E_2\cap E'_1$ in $E_2$.
             Here, $E'_1$ is the proper transform of $E_1$ on $A_2$.
                         If there is no bad curve on $E_2$, then we account it as  the inequality  automatically holds.
 \end{enumerate}   
 
     We say that a pair $(A,\a)$ is {\sl general} if the $\bR$-ideal $\a$ is general for a prime divisor computing 
     $\mld(0;A,\a)$.
     Here, the  weak transform $\a_{iA_2}$ of an ideal $\a_i\subset \o_A$ on $A_2$ is defined as
     $$\a_i\o_{A_2}=\a_{iA_2}\o_{A_2}(-v_{E_1}(\a_i)E_1-v_{E_2}(\a_i)E_2).$$
     The weak transform $\a_{A_2}$ of an $\bR$-ideal $\a$ on $A$  is defined as the canonical extension of the one for an 
     ideal of $\o_A$ (see, for example \cite{kawk1}).

 \end{defn}
 
 \begin{rem} In (2), we assume smoothness of the center  $P$ of $E$ on $A_1$.
      But it turns out that it always holds by Lemma \ref{2nd}.
 \end{rem}
 
 \begin{rem}
    The definition of generality of an $\bR$-ideal is rather complicated.
  However, one can see that under a fixed exponent, the inequalities of orders at specific curves of $E_1$ and $E_2$
  are open conditions in the space of regular functions of $A$, 
  which is the reason why we call the ideal $\a$ ``general''.
    The following gives a sufficient condition for generality of the ideal.
    
    Under the same symbols as in Definition \ref{defofgeneral}, 
     the $\bR$-ideal $\a$ is general for $E$ if one of the following hold:
    \begin{enumerate}
 \item[(1)] There is no bad curve on $E_1$ or $E_2$.
 
 \item[(2)] Assume the bad curves $B\subset E_1$ and $B'\subset E_2$ exist.
  $\ord_B\a_{A_1}\o_{E_1}=0$, and    $\ord_{B'}\a_{A_2}\o_{E_2}=0$.
  \end{enumerate}   
 
 \end{rem}


\section{Proofs of the main results}

   For the proofs of the main theorems we need the following lemma which guarantees that the second  weighted 
   blow-up is possible.

\begin{lem}\label{2nd} 
  Let $E$ be a prime divisor over a smooth $N$-fold $A$ ($N\geq 2$) with the  center at the closed point $0$.
 Let 
 $\{x_1,\ldots, x_N\}$ be a RSP at $0$.
 Let
  $v_i:=v_E(x_i)$, $v:=(v_1,\ldots,v_N)$ and define 
 $$v':=({v_1}',\ldots, {v_N}')=\frac{(v_1,\ldots,v_N)}{\gcd v}$$
 Let
 $\varphi_1: A_1\to A$ be the weighted blow-up with respect to $\{x_1,\ldots, x_N\}$ 
 with weight $v'$.
 Denote the exceptional divisor of $\varphi_1$ by 
 $E_1$.
 Assume $E\neq E_1$ and let
 $C$ be the center  of $E$ on $A_1$ and $P\in C$ the generic point of $C$.
 
 Then, $$P\in E_1 \setminus \left\{\bigcup (X_i=0)\right\}\subset E_1=\bP(v'_1, \ldots, v'_N),$$
 where $X_i$ is a homogeneous coordinate function corresponding to $x_i$.
  In particular, $P $ is  smooth on $A_1$ and also on $E_1$.
\end{lem}
\begin{proof} Assume that the statement does not hold, then we may assume that $P$ is in the hyperplane defined 
    by $X_1=0$ in $E_1=\bP(v')$.
    There exists at least one homogeneous coordinate function $X_i$ such that $P$ does not lay in the hyperplane
    defined by $X_i=0$.
    Then we obtain
 \begin{enumerate}
  \item[] $v_E(x_i)=v_{E_1}(x_i)\cdot v_E(E_1)=v'_i\cdot v_E(E_1)$.
  \item[] $v_E(x_1)=v_{E_1}(x_1)\cdot v_E(E_1)+\ord_P X_1\geq v'_1\cdot v_E(E_1)+1.$
  \end{enumerate} 
This is a contradiction to the fact that 
$$v_E(x_1):v_E(x_i)= v'_1:v'_i.$$
\end{proof}

   The following lemma is a basic idea appeared in \cite{kawk1}.
   
   \begin{lem}\label{ord}
      Let $\a$ be an $\bR$-ideal on $A$ with $a(E;A, \a)\geq 0$.
      Let $A'\to A$ be a proper birational morphism with normal $A'$, and $D$ 
     an irreducible divisor  on $A'$ with the same center on $A$ as that of $E$.
    Assume $a(D;A, \a) > a(E;A, \a)$
    and the generic point $P$ of the center of $E$ on $A'$ is smooth
    and not contained in the other exceptional divisors for $A'\to A$.
    
    Then, we have 
    $$\mld(P; D, \a_{A'}\o_{D})<0, \ \ \mbox{in\ particular}$$
    $$\ord_P\a_{A'}\o_{D}>1,$$
    where $\a_{A'}$ is a weak transform of $\a$ on $A'$. 
     \end{lem}
 
 \begin{proof} 
        First we express the log discrepancy at $E$ as follows:
\begin{equation}\label{expression}
 \begin{array}{ll}
  a(E;A, \a)&=k_{E/A}+1-v_E(\a)\\
                     &=k_{E/{A'}}+k_{D/A}\cdot v_E(D)+1-v_{D}(\a)\cdot v_E(D)
                       -v_E(\a_{A'})\\
                      &=a(E; A', I_{D}\cdot\a_{A'})+v_E(D)\cdot a(D; A,\a),\\ 
 \end{array}
 \end{equation}
   where $k_{E/A'}$ is the coefficient of the relative canonical divisor
    $K_{{\widetilde A}/{A'}}$ at $E$
   and $I_{D}$ is the 
   defining ideal of $D$ in $A'$.
   Then, by the assumption, it follows
   $a(E; A',I_{D}\cdot \a_{A'})<0$ and therefore we obtain 
   $$\mld(P; A', I_{D}\cdot \a_{A'})=-\infty.$$
   By Inversion of Adjunction (\cite{EM}, \cite{mj-p}) we obtain
   $\mld(P; D, \a_{A'}\cdot\o_{D})=-\infty$.
   Hence, it follows $\ord_P (\a_{A'}\cdot\o_{D})>1$ as claimed.
 \end{proof}

 {\bf  Setting for the proof of Theorem \ref{main}.}
  
   Let $E$ be a prime divisor over a smooth $3$-fold $A$ with the center at a closed point $0$.
   Let $\a$ be a general $\bR$-ideal on $A$  such that $a(E; A, \a)\geq0$.
  Let $$\varphi_1: A_1 \to A$$ 
  be a  squeezed blow-up for $E$ satisfying the condition (1) in Definition \ref{defofgeneral}.
  Let the squeezed system $\{x_1,x_2,x_3\}$ and 
   the weight $v'=(v_1',v'_2,v'_3)$ correspond to the squeezed blow-up $\varphi$
  (note that $v_1'=v_{2}'$).
  Denote the exceptional divisor for $\varphi$ by $E_1$.
   If $a(E_1;A, \a)\leq a(E;A, \a)$, then $E_1$ is the required prime divisor $F$ in the theorem.
 Therefore, from now on, we assume that  the inequalities $a(E_1;A, \a) > a(E;A, \a)\geq 0$ hold.

 \begin{lem}\label{E1}  
     Let $A$, $E$ and $ E_1$ be  as above.
       If $\a$ is general for $E$ and  the inequalities $a(E_1;A, \a)> a(E;A, \a)\geq 0$ hold, 
        then we obtain the following: 
    \begin{enumerate}
    \item[(i)] $0<a(E_1;A,\a)<1$.
    \item[(ii)]  $v'=(1,1,n) $ with $n\geq 1$ or $v'=(2,2,3)$.
    \begin{enumerate}
      \item In case $(1,1,n) $ the center of $E$ on $A_1$ is a curve  in $E_1=\bP(1,1,n)$ of degree $n+1$.
      \item In case $(2,2,3)$ the center of $E$ on $A_1$ is either a curve of degree 6 or a closed point in $E_1=\bP(2,2,3)$.

    \end{enumerate}

    \end{enumerate}

 \end{lem}
 
 \begin{proof}
 Let $f^e=f_1^{e_1}\cdots f_r^{e_r}\in \a$ be a general element, {\it i.e.,} $v_{E_1}(\a)=\sum_i e_i\cdot 
   \deg_{v'}(\ini_{v'} f _i)$, where $\ini_{v'} f$ is the initial part of $f$ with respect to the weight $v'$.

    We divide the proof into two cases according to the dimension of the center of $E$ on $A_1$.
    Let $P\in A_1$ be the generic point of the center  of $E$ on $A_1$.
   
 {\bf  Case 1.  $\dim\overline{\{P\}}=1$}.
 
   Let $C:=  \overline{\{P\}}$ defined by $\ell=0$ on $E_1=\bP(v')$, where $\ell$ is homogeneous of degree $\geq v'_1$
   with respect to the weight $v'$.

   The $\bR$-divisor on $E_1$ induced from a general element $f^e=f_1^{e_1}\cdots f_r^{e_r}$ is expressed as follows:
   $$\left(\prod \ini_{v'}f_i^{e_i}\right)=\alpha C +\sum_j \gamma_j C_j, \ \ \mbox{with}\ \ \alpha >1, \gamma_i \in \bR_{>0}$$ 
   Here, note that $\alpha>1$ follows from Lemma \ref{ord}.
   As $\a$ is general, $C$ is not a bad curve, therefore  its degree is greater than $v'_1$.
   Then, $\deg_{v'}\ell\geq v'_1v'_3$, because $\ell $ is an irreducible weighted homogeneous polynomial in $x_1,x_2,x_3$ of 
   weight $v'_1, v'_1,v'_3$  not contained in the coordinate hyperplanes in $E_1\simeq \bP(v')$.
   (Note that such a polynomial with smallest degree  is in the form $ax_1^{v'_3}+ bx_2^{v'_3}+c{x_3}^{v'_1}$.)
   Then, we have:
    $$v_{E_1}(\a)=\sum_i e_i\cdot 
   \deg_{v'}(\ini_{v'} f _i)=\deg_{v'}(\alpha C+\sum_j \gamma_j C_j)>\deg_{v'}C=\deg_{v'}\ell\geq v'_1v'_3.$$
   By the assumption $a(E_1;A,\a)> a(E; A,\a)\geq 0$, it follows 
\begin{equation}\label{<1}   
     0\leq a(E_1;A, \a)=2v'_1+v'_3-v_{E_1}(\a)<2v'_1+v'_3-v'_1\cdot v'_3.
\end{equation}
   The possibilities of $(v'_1,v'_1,v'_3)$ are only $(1,1,n)$ with $n\in \bN$ and $(2,2,3)$.
   In case $(2,2,3)$, by (\ref{<1}) 
    we have $a(E_1;A, \a))<2\cdot2+3-2\cdot3=1$.
    Then, in this case we have (i) and (b) of (ii).
    
    In case $(1,1,n)$ for $n\in \bN$, we have $\deg_{v'}\ell\geq n+1$.
    Indeed, if not, we have $\deg_{v'}\ell=n$ and $\ell=X_3+h(X_1,X_2)$ for a nonzero homogeneous polynomial 
    $h$ of degree $n$.
    As $E$ has the center at the curve $\ell=0$, in the same way as the proof of (\ref{oftenuse}) we have 
    $$v_E(x_3+h(x_1,x_2))>v_E(x_3),$$
    and also $x_3+h(x_1,x_2)\in \m_0\setminus \m_0^2$
    which is a contradiction to the maximality of $v_E(x_3)$.
    Therefore, in this case also we have $a(E_1;A, \a))<2+n-(n+1)=1$, which shows (i) and (a) of (ii).

  {\bf  Case 2. $\dim \overline{\{P\}}=0$}    
  
   We can take $P=(1: a: b)\in E_1=\bP(v')$ $(a,b\neq 0)$ as the homogeneous coordinate of the point $P$
   by Lemma \ref{2nd}.
   
   First we will show that $v'_1\neq 1$.
  To see this, assume that $v'_1=1$.
  Then a curve $bX_1^{v'_3}-X_3=0$ contains $P$, therefore 
  $$v_E(bx_1^{v'_3}-x_3)>v_E(x_3)=v_3,$$
  and also $bx_1^{v'_3}-x_3\in \m_0\setminus \m_0^2$ which is a contradiction to the maximality of  $v_E(x_3)$.
  
  Now we may assume that $v'_1\geq 2$.
  Then, of course $v'_1< v'_3$ and the curve $B$ defined by $aX_1-X_2=0$ contains $P$.
  Note that $B$ is the bad curve.
  
   Take a general element $f^e=f_1^{e_1}\cdots f_r^{e_r}\in \a$ 
   such that $v_{E_1}(\a)=v_{E_1}(f^e)=\deg_{v'}(\ini_{v'} f^e)$. 
  The $\bR$-divisor on $E_1=\bP(v')$ induced from a general element $f^e=f_1^{e_1}\cdots f_r^{e_r}$ is expressed as follows:
 \begin{equation}\label{generalelement}  
   \left(\prod \ini_{v'}f_i^{e_i}\right)=\alpha B +\sum_j \gamma_j C_j, \ \ \mbox{with}\ \ \alpha , \gamma_i \in \bR_{>0} 
\end{equation}
  By generality of $\a$, we have $\alpha\leq 1$.
   By Lemma \ref{ord}, we have $\mld(P; E_1, \a_{A_1}\o_{E_1})=-\infty$.
   By the description (\ref{generalelement})  of the divisor defined by a general element $f^e$,
   we have $$-\infty=\mld(P; E_1, \a_{A_1}\o_{E_1})=\mld(P;E_1,I_B^\alpha\cdot\prod_iI_{C_i}^{\gamma_i})\geq
   \mld(P;E_1,I_B\cdot\prod_iI_{C_i}^{\gamma_i})=\mld (P; B, (\prod_iI_{C_i}^{\gamma_i})\o_B).$$
    Hence, it follows $\ord_P  (\prod_iI_{C_i}^{\gamma_i})\o_B>1$.
    Applying Lemma \ref{bezout} to the curve $B$ of degree $v'_1$, we obtain
    $$1<\ord_P  (\prod_iI_{C_i}^{\gamma_i})\o_B\leq \frac{\sum \gamma_i \deg_{v'}C_i}{v'_1v'_3}\leq
   \frac{v_{E_1}(f^e)}{v'_1v'_3}\leq \frac{2v'_1+v'_3}{v'_1v'_3},$$
   Here, for the third inequality, we use 
    $$\sum\gamma_i\deg_{v'}C_i\leq v_{E_1}(f^e)-\alpha v_1'.$$
    Then, the only possibility of $v'$ satisfying these inequalities is $(2,2,3)$ and we also have $v_{E_1}(\a)=v_{E_1}(f^e)> 2\cdot 3$
    which completes the proof of (i) and (ii) in case $\dim \overline{\{P\}}=0$.
   
 \end{proof}
 
 \begin{cor}[Theorem \ref{mld1}] Let $A$ be a smooth variety of dimension 3 over an algebraically closed field $k$.
 For any general pair $(A,\a)$ with $\mld(0; A, \a)\geq 1$ 
  the minimal log discrepancy is computed by a prime divisor obtained by
   one weighted blow-up.
 \end{cor}
  
\begin{proof} As $a(E_1;A,\a)\geq \mld(0; A,\a)\geq 1$,  
 the inequality $a(E_1;A, \a)> a(E;A, \a)$ does not hold by (i) in 
  Lemma \ref{E1}.
\end{proof}
  
  
  {\bf Proof of Theorem \ref{main}.}
  
  Let $A_1$, $E_1$ be as in the setting above.
  Assuming  $0\leq a(E;A,\a)< a(E_1;A,\a)$, 
  we will prove that $a(E;A,\a)\geq a(E_2;A,\a)$ for a divisor $E_2$ obtained by
  the second ``blow-up" constructed below in Case 1 and Case 2.
  
  Let $P\in E_1\subset A_1$ be the center of $E$.
First,  for every prime divisor $D$ over $A_1$ with the center at $P$
 and with the inequality $a(D;A,\a)>a(E; A,\a)\geq0$,
 we observe that
 \begin{equation}\label{a1}
 a(D;A_1, \a_{A_1})\geq 0. 
  \end{equation}
  Indeed, we have an expression of $a(D; A,\a)$ as follows:
  $$a(D; A,\a)=a(D; A_1,\a_{A_1})+v_D(E_1)(a(E_1;A,\a)-1).$$
  As $a(D; A,\a)\geq0$ and  $a(E_1;A,\a)-1<0$ ( Lemma \ref{E1} ),
  we have $a(D; A_1,\a_{A_1})\geq0$.

  {\bf Case 1.  $\dim \overline{\{P\}}=1$}

   Let  $\{y_1,y_2\}$ be a 
    squeezed system for $E$ on $A_1$ at $P$ and
  $E_2$ the prime divisor obtained by the   squeezed  blow-up of $A_1$ at $P$
  with respect to $\{y_1,y_2\}$.
  Let $K:=\o_{A_1,P}/\m_{A_1,P}$ and $\ok$ the algebraic closure of $K$.
   Let $A_{1K}:=\spec \widehat\o_{A,P}$, $A_{1\ok}:=\spec \ok\widehat\o_{A,P}=
   \spec \ok[[y_1,y_2]]$.
   Denote the both closed points of $A_{1K}$ and of $A_{1\ok}$ by $0$.
   Here, we note that $\{y_1,y_2\}$ is not necessarily a squeezed system on $A_{1\ok}$ for $\oe$ as is shown
   in Example \ref{not.stan.exmp}, but it does not matter.
   Because we are interested only in ideals which came from $A_1$ and in this case a squeezed system on $A_1$
   for $E$  works
   in the same way as in \cite{kawk1} and \cite{ip}, which one can see below:
 
   Let $\tilde A \to A_1$ be a log resolution of $(A_1, \a\o_{A_1})$ on which $E$ appears.
   Then, the base change $\tilde{\tilde A}\to A_{1\ok}$ by $A_{1\ok}\to A_1$ is also a 
   log resolution of $(A_{1\ok}, \a\o_{A_{1\ok}})$ on which the prime divisor $\oe$ corresponding to $E$ appears.
   Let $A_2\to A_1$ be the squeezed blow-up with respect to the squeezed system $\{y_1,y_2\}$
   and $E_2$ the exceptional divisor.
   By definition, it means that 
    $A_{2\ok}\to A_{1\ok}$ is squeezed weighted blow-up with respect to the squeezed system $\{y_1,y_2\}$
    and $\oe_2$ be the exceptional divisor corresponding to $E_2$.

     If $\oe=\oe_2$, then we have $E=E_2$ and  we are done.
     So, we may assume that the center of $\oe$ on $A_{2\ok}$ is a point.
     Then the center $Q\in A_{2\ok}$ is not on the proper transform of $\oe_1$ on $A_{2\ok}$.
     This is proved as follows:
     
     Let $w=(r,s)$ be the weight of the  squeezed system $\{y_1,y_2\}$ on $A_{1}$.
        
          First, we show that $r=s$ does not happen.
     Assume $r=s$, {\it i.e.,} $w=(1,1)$, then we can take an expression $Q=(a,b)$  of $Q\in \oe_2=\bP_\ok^1$ by homogeneous coordinates with
     $a, b\neq 0$.
     Let $z:=by_1-ay_2\in \o_{A_{1\ok}}$.
     As $Q$ is the center of $\oe$ on $\oe_2\subset A_{2\ok}$ and satisfying $bY_1-aY_2=0$ 
     ($Y_1,Y_2$ are
     the homogeneous coordinates on $E_2=\bP_K^1$ corresponding to $y_1,y_2$.), it follows 
     $$z\in \m_Q\setminus \m_Q^2, \ \ \ \mbox{and}\ \ \ v_E(z)>v_E(y_1), v_E(y_2),$$
       which is a contradiction to the fact that $\{y_1,y_2\}$ is a  squeezed system.
      Now, we may assume that $r<s$. 
      Let $h=0$ be the defining equation of $E_1$ in $A_1$ around $P$, then $\oe_1$ is also defined by $h=0$
      and it is smooth at the closed point $0\in A_{1\ok}$.
      Therefore, we have $\ord_{y_1,y_2}h=1$.
      Then the initial part  of $h$ with respect to $w$ is one of the following:
      
      (1) $\ini_w(h)=y_1$,\ \ \ (2) $\ini_w(h)=y_2$, \ \ \  (3) $\ini_w(h)=y_2+a{ y_1}^d$  ( $a\in \ok$, $w_1d=w_2$).
      In the first two cases, $\oe'_1\mid_{\oe_2}$ is in the zero locus of the coordinate functions,
      where $\oe'_1$ is the proper transform of $\oe_1$ on $A_{2\ok}$.
      Therefore it does not contain the center $Q$ of $\oe$ by Lemma \ref{2nd}.
      In case (3), it follows $w=(1,d)$.
      If  $Q$ is in $\oe'_1\mid_{\oe_2}$, then we have $y'_2:=y_2+a{ y_1}^d\in \m_{A_{1\ok},0}\setminus \m_{A_{1\ok},0}^2$ and
      $v_\oe(y'_2)>v_\oe(y_2)$ which is a contradiction to the assumption that $\{y_1,y_2\}$ is a  squeezed system.
         Now, in any case we obtain that $Q\not\in \oe'_1$.

  On the other hand, $a(E; A,\a)$ has another expression as follows:
   $$a(E; A,\a)=k_{E/A_1}+k_{E_1/A}\cdot v_E(E_1)+1-v_E(\a).$$  

 It is sufficient to show that
  $$a(\oe; A, \a)\geq a(\oe_2;A,\a).$$
  Assume contrary, then
  \begin{equation}\label{Qnot}
  0> \oa(\oe; A, \a)- \oa(\oe_2;A,\a)=a(\oe; A_{2\ok},I_{\oe_2}\cdot\a_{A_{2\ok}})+(v_{\oe}(\oe_2)-1)\cdot\oa(\oe_2; A,\a),
  \end{equation}
  where $\a_{A_{2\ok}}$ is the weak transform of $\a_{A_1}\o_{A_{1\ok}}$.
  For the calculation of (\ref{Qnot}), we used
  \begin{enumerate}
 \item[$\bullet$] $v_\oe(\oe_1)=v_\oe(\oe_2)v_{\oe_2}(\oe_1)+v_\oe(\oe'_1)=v_\oe(\oe_2)v_{\oe_2}(\oe_1)$.
\end{enumerate}
    Then the inequality (\ref{Qnot}) shows that 
 $a(\oe; A_{2\ok},I_{\oe_2}\cdot\a_{A_{2\ok}})<0$ which implies 
   $$\mld(Q; A_{2\ok},I_{\oe_2}\cdot\a_{A_{2\ok}})=-\infty.$$
   Then, by Inversion of Adjunction  (\cite{EM}, \cite{mj-p}), it follows
   $$\mld(Q; \oe_2, \a_{A_{2\ok}}\cdot\o_{\oe_2})<0$$
   which yields $\ord_Q((\a_{A_1}\o_{A_{1\ok}})_{A_{2\ok}}\cdot\o_{\oe_2})=\ord_Q(\a_{A_{2\ok}}\cdot\o_{\oe_2})>1$

Let $(r,s)$ be the  squeezed weight for $\oe$ at the closed point $0\in A_{1\ok}$, then 
$$a(\oe, A_{1\ok},\a_{A_{1\ok}})=a(E; A_1, \a_{A_1})\geq0,$$
where the last inequality follows from (\ref{a1}).
   Now we reach the situation in Theorem \ref{prototype} and apply the
   argument in (\cite{kawk1}) for the surface pair $(A_{1\ok},\a_{A_{1\ok}}) $, 
we obtain
\begin{equation}\label{rs}
1<\ord_Q((\a_{A_1}\o_{A_{1\ok}})_{A_{2\ok}}\cdot\o_{\oe_2})\leq \frac{v_{\oe_2}(\a_{A_1}\o_{A_{1\ok}})}{r\cdot s}\leq \frac{r+s}{r\cdot s},
\end{equation}
where note that $\a_{A_{2\ok}}=(\a_{A_1}\o_{A_{1\ok}})_{A_{2\ok}}$  
and
the third inequality follows from 
$$r+s-v_{\oe_2}(\a_{A_1}\o_{A_{1\ok}})=
a(\oe_2; A_{1\ok}, \a_{A_1})= a(E_2; A_1, \a_{A_1})\geq 0$$ by (\ref{a1}).
       The possible positive intergers $\{r,s\}$ satisfying (\ref{rs}) with $\gcd(r,s)=1$ are only $\{1,s\}$.
       In this case let $z':= y_1^s - cy_2$, where $Q=(c,1)\in \oe_2=\bP(1,s)$, then $v_\oe(z')> v_\oe(y_2)$, which is a contradiction to 
that $\{y_1, y_2\}$ is a  squeezed system for $\oe$.
Hence we obtain  
$$\oa(\oe; A, \a)\geq \oa(\oe_2;A,\a),$$
which completes the proof of the theorem for Case 1.

{\bf Case 2. $\dim\overline{\{ P\}}=0$}

   Since we are assuming  $0\leq a(E;A,\a)<a(E_1; A,\a)$, by Lemma \ref{E1} only possibility of $v'$ is $(2,2,3)$
   and we have $0\leq a(E_1;A,\a)<1$.


  
   Now take a  squeezed blow-up $A_2 \to A_1$ of weight $w=(w_1,w_2,w_3)$ at $P$ and let $E_2$ be the exceptional divisor.
   We may assume that the condition (2) in Definition \ref{defofgeneral} holds.
   Let $Q\in E_2$ be the center of $E$ on $A_2$.


      Let $E'_1$ be the proper transform  of $E_1$ on $A_2$.
      Denote the defining ideals of $E'_1$ and $E_2$ in $A_2$ by $I_{ E'_1}$ and $I_{E_2}$, respectively.
     
      Then, we have  the similar expansion of $a(E;A,\a)$ as in (\ref{expression}) as follows:
\begin{equation}\label{noE1}
      a(E; A,\a)=  a(E; A_2, I_{E'_1}\cdot I_{E_2}\cdot\a_{A_2}) +v_E(E_2)a(E_2; A, \a)+v_E(E'_1)a(E_1; A, \a),      
\end{equation}      
    where $\a_{A_2}$ is the weak transform of $\a$ on $A_2$ and is also the weak transform of $\a_{A_1}$ on $A_2$.

 {\bf Case 2.1.}    $\dim\overline{\{Q\}}=0$:
   
   We will prove $a(E_2;A,\a)\leq a(E;A,\a)$.
   Assume on the contrary that $a(E_2;A,\a)> a(E;A,\a)$.
   Then, by (\ref{noE1}), we obtain 
   \begin{equation}\label{IA}
   a(E; A_2, I_{E'_1}\cdot I_{E_2}\cdot\a_{A_2}) <0
 \end{equation}
  It implies that  $\mld(Q; A_2,   I_{E'_1}\cdot I_{E_2}\cdot\a_{A_2})=-\infty$.
  Let $ L:=  E'_1\cap E_2$, by Inversion of Adjunction, we obtain  
   $$\mld(Q;E_2, I_L\a_{A_2}\o_{E_2})<0.$$
    Let $B'$ be the bad curve on $E_2$ (note that a bad curve exists in our case by Lemma \ref{non-exist-bad}).
    Then, we obtain  
    \begin{equation}\label{case2.1}
    \ord_{B'}\a_{A_2}\o_{E_2}\leq1.
    \end{equation}
   Indeed,
   when $L=B'$, then generality of $\a$ implies that $\ord_{B'}\a_{A_2}\o_{E_2}=0$, as $\ord_{B'}I_L=1$.
    On the other hand, when $L\neq B'$, then $Q\not\in L$ and therefore generality implies 
    $\ord_{B'}\a_{A_2}\o_{E_2}\leq1$.
  Now, in the same way as Case 2 in the proof of Lemma \ref{E1}, we obtain that the weight of the second
   squeezed blow-up is $(2,2,3)$.

   We will show a contradiction under this situation.  
    In this case, we have
 \begin{equation}\label{6}  
   v_{E_2}(\a_{A_1})>6, \ \ \mbox{as\ well\ as\ \ \ }v_{E_1}(\a)>6,
\end{equation}   
  by applying (i) of Lemma \ref{E1} for $(A_1, \a_{A_1})$, $ E_2$ with the weight  $w=(2,2,3)$
  and also for $(A, \a)$, $ E_1$ with the weight $v'=(2,2,3)$. 
   As the  squeezed system $\{y_1,y_2,y_3\}$ at $P\in A_1$ has weight $(2,2,3)$, it follows 
   $v_{E_2}(f)\leq 3\cdot \ord_P f$ for every $f\in \a_{A_1}$. 
   Therefore we obtain
   \begin{equation}\label{v&ord}
   v_{E_2}(\a_{A_1})\leq 3\cdot \ord_P\a_{A_1}      \leq 3\cdot \ord_P\a_{A_1}\o_{E_1}.
   \end{equation}

   On the other hand, applying Lemma \ref{bezout} to $E_1=\bP(2,2,3)$ and a general element of $\a_{A_1}\cdot\o_{E_1}$,
   we obtain $1<\ord_P\a_{A_1}\o_{E_1}\leq \frac{v_{E_1}(\a)}{2\cdot 3}$. Note that the first inequality follows from
   Lemma \ref{ord}.

   Then, it follows
\begin{equation}\label{7}
   7=2+2+3=k_{E_1}+1\geq v_{E_1}(\a)\geq 6\cdot \ord_P\a_{A_1}\o_{E_1}.
\end{equation}   
   Using   (\ref{6}), (\ref{v&ord}) and (\ref{7})
   we obtain
   $$\frac{7}{2}>3\cdot \ord_P\a_{A_1}\o_{E_1}\geq v_{E_2}(\a_{A_1})>6$$
   which is a contradiction. 
   Therefore $a(E_2;A,\a)\leq a(E;A,\a)$ holds.
 
 {\bf Case 2.2.}  $\dim \overline{\{Q\}}=1$.
 
   In the following, we will prove  $a(E_2;A,\a)\leq a(E;A,\a)$.
 Assume contrary,  $a(E_2;A,\a)>a(E;A,\a)$.
 The curve $\overline{\{Q\}}$ is not a bad curve, because if it is, then
   $$-\infty=\mld(Q;A_2, I_{E'_1}\cdot I_{E_2}\cdot\a_{A_2})=\mld(Q; E_2, I_L\a_{A_2}\o_{E_2})$$ 
   implies 
   $\ord_QI_L\a_{A_2}\o_{E_2}>1$, while the generality of $\a$ implies the converse inequality
   $\ord_QI_L\a_{A_2}\o_{E_2}=\ord_{B'}I_L\a_{A_2}\o_{E_2}\leq 1$.
   We also have $\overline{\{Q\}}\neq L$. 
   This is proved as follows: 
   
   Let $h'\in \o_{A_1}$ define $E_1$ around $P$.
   As $P$ is smooth on $E_1$ and also on $A_1$, we have $\ord h'=1$ with respect to RSP $\{y_1,y_2,y_3\}$ of $\o_{A_1}$
   at $P$. 
   Then, considering of the initial term of $h'$ with respect to the weight $w$, we see that one of the following holds:
\begin{enumerate}
   \item[(1)]     $L$ is a coordinate axis of $E_2=\bP(w)$;
   \item[(2)]  $L$ is defined by $Y_1+aY_2$  $(a\in k)$ in $E_2$;
   \item[(3)]  $L$ is defined by $Y_3+f(Y_1,Y_2)$ in $E_2$, where $f$ is a homogeneous polynomial of degree $d$.
\end{enumerate}
    In the third case, the weight $w$ must be $(1,1,d)$.
    In this case, if $\overline{\{Q\}}= L$,  it follows $y'_3:=y_3+f(y_1,y_2)\in \m_{A_1,P}\setminus \m_{A_1,P}^2$ and 
    $v_E(y'_3)> v_E(y_3)$, which is a contradiction to the maximality of $v_E(y_3)$.
    In case (1), $\overline{\{Q\}}\neq L$ because $Q$ is not contained in the coordinate axes (Lemma \ref{2nd}).
    In case (2), $L$ becomes the bad curve, therefore $\overline{\{Q\}}\neq L$, because $\overline{\{Q\}}$ is not the bad curve, 
    as we saw above.

  Now we obtain $Q\not\in E'_1\cap E_2$.  
  By using  this, we have 
  $$\mld(Q;A_2, I_{E_2}\cdot\a_{A_2})=\mld(Q;A_2, I_{E'_1}\cdot I_{E_2}\cdot\a_{A_2})=-\infty.$$
  By Inversion of Adjunction, we have 
  $$\mld(Q;E_2,\a_{A_2}\o_{E_2})=-\infty.$$
  Then, we have $1< \ord_Q\a_{A_2}\cdot\o_{E_2}$

First we show that the  squeezed weight $w=(r,r,s)$ for $E$ at $P\in A_1$ is $(1,1,n)$ for $n\in \bN$.
Let $C:=\overline{\{Q\}}$ be defined by $\ell=0$ in $E_2=\bP(r,r,s)$.
If $w\neq (1,1,n)$, then the other possible weight $w$ is $(2,2,3)$.
    In this case 
the smallest possible value for the  degree of $\ell$  on $\bP(2,2,3)$ with respect to $w$ is $6$.
Therefore, by $1< \ord_Q\a_{A_2}\cdot\o_{E_2}$, 
$$v_{E_2}(\a_{A_1})\geq \deg_w\ell\cdot\ord_Q(\a_{A_1})_{A_2}\geq 6\cdot\ord_Q(\a_{A_1})_{A_2}>6$$
Now we obtain the inequality (\ref{6}).
The inequalities (\ref{v&ord}) and (\ref{7}) also hold in the present case.
Therefore, we induce a contradiction and $w$ must be $(1,1,n)$. 
By Lemma \ref{E1}, $deg_{w}\ell=1+n$.

Let $\{y_1,y_2,y_3\} $ be a  squeezed system at $P\in A_1$ with the weight $(1,1,n)$.
Let $\{Y_1,Y_2,Y_3\}$ be the homogeneous coordinates of $E_2=\bP(1,1,n)$ corresponding to $\{y_1,y_2,y_3\} $.
As $\ell$ is irreducible of degree $1+n$ with respect to the weight $(1,1,n)$,
 we can express
 $$\ell=Y_1Y_3-Y_2^{n+1}.$$
 For simplicity, assume $\a=\a_1^{e_1}$ and take a general element $f\in  \a_1\o_{A, 0}\subset k[[x_1,x_2,x_3]]$, 
 where $\{x_1,x_2,x_3\}$ is a  squeezed system for $E$ at $0\in A$ of weight $(2,2,3)$.
 Then the weak transform $f_{A_1}$ of $f$ on $A_1$ is written as
 \begin{equation}\label{fA1}
 f_{A_1}=(y_1\cdot y_3-y_2^{n+1})^r\cdot \ell' + g(y),
 \end{equation}
 where $\ell'$ is weighted homogeneous  and $g(y)$ is the  term with the higher weight with respect to the weight $w=(1,1,n)$.
 
 Here, we may assume that $P=(1,1,1)\in E_1=\bP(2,2,3)$,
 then we can take a RSP at $P\in A_1$ by making use of the  squeezed system $\{x_1,x_2,x_3\}$ of  squeezed weight 
 $(2,2,3)$ which gives the first weighted blow-up $\varphi_1:A_1\to A$:
 $$z_1=\frac{x_1^3-x_3^2}{x_3^2}, \ \ \ z_2=\frac{x_2^3-x_3^2}{x_3^2}, \ \ \ z_3=x_3,$$
 where $x_3 $ defines $E_1$ in the neighborhood of $P$.
Take the minimal  $m\in \bN$ such that
\begin{equation}\label{erasedenominator}
f=x_3^m\cdot f_{A_1}\in \o_{A,0}\subset k[[x_1,x_2,x_3]].
\end{equation}
We note that for $m\geq 2$, 
\begin{equation}\label{denominator}
   \ord_0 x_3^m\cdot z_i=m\ \ (i=1,2),\ \ \ \ \ \ \
   \ord_0 x_3^m\cdot z_3=m+1,
\end{equation}
where $\ord_0$ is the order with respect to the parameters $x_1,x_2,x_3 $ in $\o_{A,0}$.
Then, by (\ref{denominator}), $$\ord_0 f=\ord_0 (x_3^m\cdot f_A)\geq m.$$

On the other hand if $x_3^s(y_1y_3-y_2^{n+1})^r\in \o_{A,0}$, it should be $s\geq 4r$.
In fact, if a quadratic monomial  $z_iz_j$ $(i, j\in\{1,2\})$ appears in $y_1y_3$ which is expressed as a function of $z_1,z_2,z_3$,
then $s\geq 4r$. 
If such a monomial $z_iz_j$ $(i, j\in\{1,2\})$ does not appear in $y_1y_3$,  then $z_i$ $(i<3)$ appears in $y_2$,
     because  $\{z_1,z_2,z_3\}$ and 
$\{y_1,y_2,y_3\}$ are both RSP at $P\in A_1$.
This yields $s\geq 2(n+1)r\geq4r$.

       Consider the initial part $(y_1\cdot y_3-y_2^{n+1})^r\cdot \ell'$ of $f_{A_1}$ with respect to the weight $w=(1,1,n)$.
       We know that $a(E_2; A_1,\a_{A_1})\geq 0$, therefore $v_{E_2}(f_{A_1}^{e_1})=v_{E_2}(\a_{A_1}^{e_1})
       \leq k_{E_2/A_1}+1=n+2$.
       Then, it follows that 
\begin{equation}\label{e2}
       e_1(r(n+1)+\deg_{w}\ell')\leq n+2.
\end{equation}       
       As $1<\ord_Q \a_{A_2}\o_{E_2}$,  it follows $1<\ord_Q (y_1y_3-y_2^{n+1})^{re_1}$ which yields $re_1>1$.
     By this and (\ref{e2}), we have $\deg_{w}\ell'<r$, therefore $\ord_P \ell'< r$ which yields that the factor of $z_3(=x_3)$
     appears in $\ell'$ at most $r-1$ times.
     Hence, as (\ref{erasedenominator}) the inclusion $x_3^m (y_1\cdot y_3-y_2^{n+1})^r\cdot \ell'   \in \o_{A,0}$
     should hold, which implies
     $m\geq 4r-(r-1)=3r+1$.

       Then, 
       $\ord_0 f=\ord_0 (x_3^m\cdot f_{A_1})\geq 3r+1,$ and therefore, taking $e_1r>1$ into account, we have
	$$\ord_0\a_1^{e_1}=\ord_0 f^{e_1}\geq  e_1(3r+1)>3.$$
	Then, for every prime divisor $D$ over $A$ with the center at $0$ has the discrepancy 
	$a(D;A,\a)<0$, which is a contradiction to the condition that $a(E; A, \a)\geq 0$.
\ \ \ \ \ \ \ \ \ \   {\bf QED of Theorem \ref{main}.}

    The condition ``general" is necessary as far as we use `` squeezed" blow-ups to construct a required divisor
  in Theorem \ref{main}.
  Actually, we have a non-general ideal such that two  squeezed blow-ups do not give the required divisor.
   \begin{exmp}\label{3times}
    Let $f=(x_1-x_2)^2+x_3^2+x_1^4\in k[x_1,x_2,x_3]$, $e=6/5$ and $\a=(f)^e$. 
   Define $E$ as follows:
  \newline 
   g $\varphi_1:A_1\to A$ be the weighted blow-up with weight $(1,1,2)$ with respect to the coordinates $\{x_1,x_2,x_3\}$.
   Let $E_1$ be the exceptional divisor of $\varphi_1$.
   Let $\varphi_2:A_2\to A_1$ be the (usual) blow-up with the center at $E_1\cap (f_{A_1}=0)$, where $(f_{A_1})$ is the 
   weak transform of $(f)$ on $A_1$.
   Let $E_2$ be the exceptional divisor of $\varphi_2$.
   Let $\varphi_3:\widetilde A\to A_2$ be the (usual) blow-up with the center at $E_2\cap (f_{A_2}=0)$, where $(f_{A_2})$ is the 
   weak transform of $(f)$ on $A_2$.
   Let $E$ be the exceptional divisor of $\varphi_3$.
   Then, $\varphi_1$ and $\varphi_2$ are  squeezed blow-ups for $E$, $\a$ is not general for $E$ and the following hold:
   $$0=a(E;A,\a)<a(E_2;A,\a)=\frac{1}{5}<a(E_1;A,\a)=\frac{3}{5}.$$
   So, we can see that the  squeezed blow-ups do not work for this ideal.
   But if we do not stick to  squeezed blow-up, we can find two weighted blow-ups to obtain the required $F$ in the theorem.
   Let $\{x'_1, x'_2,x'_3\}$ be another  RSP defined by $x'_i=x_i$ $(i=1,3)$ and $x'_2=x_1-x_2$.
   Then, $v_E(x'_1)=1,v_E( x'_2)=2,v_E(x'_3)=2$. (We can see that this RSP is not  squeezed.)
   Now, let $\psi_1:A'_1\to A$ be the weighted blow-up with weight $(1,2,2)$ with respect to  $\{x'_1, x'_2,x'_3\}$.
   Let $E'_1$ be the exceptional divisor of $\psi_1$.
   Let $\psi_2:A'_2\to A'_1$ be the (usual) blow-up with the center at $E'_1\cap (f_{A'_1}=0)$.
   Let $E'_2$ be the exceptional divisor of $\psi_2$.
   Then, we can see that $E=E_2$ at the generic points.
   So, $E$ itself is obtained by two weighted blow-ups.
   \end{exmp}
   
   The example suggests us that we may take an appropriate weighted blow-up to obtain the required $F$ in the theorem,
   if $\a$ is not general.

 \begin{cor}[Corollary \ref{maincor}] 
 Assume $N=3$.
    Then, for every ``general"  pair $(A,\a)$, the minimal log discrepancy $\mld(0;A,\a)$ is
    computed by a prime divisor $E$ obtained by at most two weighted blow-ups. 
    More concretely, the blow-ups are  squeezed blow-ups for $E$.
 \end{cor}
 
 \begin{proof}
When $\mld(0;A,\a)\geq 0$, then apply the theorem for a divisor $E$ computing the mld.
     When $\mld(0;A,\a)=-\infty$, then in a similar way as in \cite{kawk1}, take a prime divisor $E$ computing 
     the mld.
     Then by taking a positive real number $t<1$ such that $a(E; A,\a^t)=0$ and apply Theorem \ref{main}.
 \end{proof}
 
 \begin{cor}\label{nobadcurve}
   Let $E$ be a prime divisor over $A$ with the center at $0$ and $E_1=\bP(r,r,s)$ $(r,s\geq 1)$ the exceptional divisor of
   a squeezed blow-up for $E$.
   Assume that $a(E;A,\a)\geq 0$ and the center of $E$ on $E_1$ is a curve of degree $>r$, then 
   there is a prime divisor $F$ such that 
   $$a(F; A,\a)\leq a(E;A,\a)$$
   holds for every $\bR$-ideal $\a$ 
   and $F$ is obtained by at most two weighted blow-ups.
 \end{cor}
 
 \begin{proof}  We can see that there is no bad curve on $E$. 
     Therefore, every $\bR$-ideal $\a$ is general for $E$.
 \end{proof}

 The proof of the theorem shows also the following corollary.
 
\begin{cor} Let $E$ be a prime divisor over $A$ with the center at $0$ computing $\mld(0;A,\a)\geq 0$.
   Let $E'$ be the exceptional divisor of a weighted blow-up with weight $v:=(r,s,t)$, where $\gcd(r,s,t)=1$.
   Assume that the center $C$ of $E$ on $E'$ is a curve of degree $d\geq r+s+t-1$
   If $\mld(0;A,\a)$ is not computed by $E'$,
   then the mld is computed by the divisor obtained by one additional weighted blow-up at $C$.
\end{cor}
 
 \begin{proof} Let $A'\to A $ be the weighted blow-up with weight $(r,s,t)$.
   By the assumption, we have $a(E;A,\a)< a(E'; A,\a)$.
   Then, by  Lemma \ref{ord}, we have $\alpha:=\ord_P\a_{A'}\o_{E'}>1$, where $P$ is the generic point of $C$.
   Therefore, we obtain $v_{E'}(\a)=\alpha d> r+s+t-1$, and therefore $a(E';A,\a)<1$.
   Now, in the same way as Case 1 in the proof of Theorem \ref{main}, we obtain that the  squeezed blow-up at $P$
   gives a divisor $F$ satisfying $a(F;A,\a)\leq a(E;A,\a)=\mld(0;A,\a)$.  
  \end{proof}
 
 The following is a special case of the corollary above. 
 Example \ref{not1} is in this case.
 \begin{cor}[Corollary \ref{usual}]\label{111}
   Let $E$ be a prime divisor over $A$ with the center at $0$ computing $\mld(0;A,\a)\geq 0$.
   Let $E'$ be the exceptional divisor of the usual blow-up with the center at $0$.
   Assume that the center $C$ of $E$ on $E'$ is a curve of degree $d\geq 2$
   Then, $\mld(0;A,\a)$ is computed by the divisor obtained by one additional weighted blow-up at $C$. 
 \end{cor}

\makeatletter \renewcommand{\@biblabel}[1]{\hfill#1.}\makeatother

\noindent   
Graduate School of Math. Sci., the University of Tokyo, \\
shihokoishii@mac.com/
shihoko@g.ecc.u-tokyo.ac.jp
\end{document}